\documentclass[12pt]{amsart}

 \newtheorem{thm}{Theorem}[section]
 
 \newtheorem{lem}[thm]{Lemma}
 
 \theoremstyle{definition}
 
 \theoremstyle{remark}
 
 \newtheorem{example}[thm]{Example}
 \usepackage{enumitem}
 \usepackage{color}

\begin{document}

\title [A note on power values of  derivation  in prime . . .]
  {A note on power values of  derivation  in prime and semiprime rings}

\author{shervin sahebi$^{1}$, venus rahmani$^{2,*}$ }
\thanks{$^{*}$ Corresponding author.  Email: ven.rahmani.math@iauctb.ac.ir}

\address{ $^{1,2}$ Department Of Mathematics, Islamic Azad University,
 Central Tehran Branch, 13185/768, Tehran, Iran.}

\keywords{Derivation, prime ring,  semiprime ring, Martindale quotient ring.}


\begin{abstract}
 let $R$ be a ring with derivation $d$,
 such that $(d(xy))^n=(d(x))^n(d(y))^n$ for all
 $x, y \in R$ and $n\geq 1$ is a fixed integer.
 In this paper, we show that
 if $R$ is a prime,
 then $d=0$ or $R$ is a commutative.
 If $R$ is a semiprime,
 then $d$ maps $R$ in to its center.
 Moreover in semiprime case let $A=O(R)$ be the
 orthogonal completion of $R$ and $B=B(C)$ be the Boolian ring of $C$,
 where $C$ is the extended centroid of $R$, then there exists
 an idempotent $e\in B$ such that $eA$ is commutative
 ring and $d$ induce a zero derivation on $(1-e)A$.\\ \\

 MSC: 16R50; 16N60; 16D60
 \end{abstract}

\maketitle

\section{Introduction}
 Let $R$ be an associative ring with center $Z(R)$.
 Recall that an additive
 map $d: R\rightarrow R$ is called derivation if
 $d(xy)=d(x)y+ xd(y)$, for all $x, y \in R$.
 Many results in literature indicate that global
 structure of a prime (semiprime) ring $R$ is
 often lightly connected to the behaviour of
 additive mappings defined on $R$.
 A well-known result of Herstein~\cite{a222} stated that
 if $R$ is a prime ring and $d$ is an inner
 derivation of $R$ such that $d(x)^n=0$ for all
 $x\in R$ and $n\geq 1$ fixed integer, then $d=0$.
 The number of authors extended this
 theorem in several ways.
 In \cite{a2} Giambruno and Herstein
 extended this result to arbitrary derivations
 in semiprime rings.
 In~\cite{a0001} Carini and Giambruno proved that if $R$
 is a prime ring with derivation $d$ such that
 $d(x)^{n(x)}=0$ for all $x\in L$, a Lie ideal of $R$,
 then $d(L)=0$ when $R$ has no non-zero nil right ideal and
 char $R\neq 2$. The same conclusion holds when
 $n(x)=n$ is fixed and $R$ is a $2$-torsion free
 semiprime ring.
 Using the ideas in~\cite{a0001} and the methods
 in~\cite{a02} Lanski~\cite{a003} removed both
 the bound on the indices  of nilpotence and the
 characteristic assumptions on $R$.
 In~\cite{a00001} Bresar gave a generalization of the
 result due to Herstein and  Giambruno~\cite{a2}
 in another direction. Explicitly, he proved
 in semiprime ring $R$ with derivation $d$
 and $a\in R$, if $ad(x)^n=0$ for all $x\in R$,
 where $n\geq 1$ is a fixed integer, then $ad(R)=0$ when
 $R$ is an $(n-1)!$-torsion free ring.
 In recent years, a number of articles discussed derivations
 in the context of prime and semiprime rings (see \cite{Carini, fili, sharma, dhara, Argac, dhara2}).

 \noindent
 But here we will extend  Herstein result's~\cite{a222}
 when the condition is more widespread.\\
 Indeed, we  consider  the situation when
 $(d(xy))^n=(d(x))^n(d(y))^n$ for all
 $x, y \in R$ and $n\geq 1$ is a fixed integer.\\

 \noindent
 The main results in this paper are as follows:
 \begin{thm}\label{b0}
 Let $R$ be a prime ring  and $d$
 a derivation of $R$. Suppose $(d(xy))^n=(d(x))^n(d(y))^n$ for all
 $x, y \in R$ and $n\geq 1$ is a fixed integer. Then $d=0$ or $R$ is
 commutative.
 \end{thm}

 \noindent
 When $R$ is a semiprime ring, we  prove:

 \begin{thm}\label{b00}
  Let $R$ be a   semiprime ring and $d$
  a non-zero derivation of $R$.
  Suppose $(d(xy))^n=(d(x))^n(d(y))^n$ for all
  $x, y \in R$ and $n\geq 1$ is a fixed integer.
  Then $d$ maps $R$ into its center.
  \end{thm}
  \begin{thm}\label{b000}
 let $R$ be a  semiprime ring with
 derivation $d$.  Consider
 $(d(xy))^n=(d(x))^n(d(y))^n$ for all
 $x, y \in R$ and $n\geq 1$ is a fixed integer.
 Further, let $A=O(R)$ be the orthogonal completion
 of $R$ and $B=B(C)$ where $C$ the extended centroid
 of $R$. Then there exists idempotent $e\in B$ such
 that $eA$ is a commutative ring and $d$ induce a
 zero derivation on $(1-e)A$.
 \end{thm}
 \noindent
 Throughout the paper we use the standard notation
 from \cite{a0}.

 \noindent
 In particular, we denote by  $Q$
 the two sided Martindale quotient of prime (semiprime) ring $R$
 and $C$ the center of $Q$. We call $C$ the extended centroid
 of $R$.\\

\section{proof of the main results}
 Firstly we consider the case when $R$ is a prime
 ring.
 The following results are useful tools
 needed in the proof of Theorem\ref{b0}.

\begin{lem}\label{b1}
\emph{(see \cite[Theorem 2]{a1})}.
 Let $R$ be a prime ring and $I$ a non-zero ideal of $R$.
 Then $I$, $R$ and $Q$ satisfy the same generalized
 polynomial identities with coefficient in $Q$.
 \end{lem}
 \begin{lem}\label{b2}
 \emph{(see \cite[Theorem 2]{a4})}.
 Let $R$ be a prime ring and $I$ a non-zero ideal of $R$.
 Then $I$, $R$ and $Q$ satisfy the same differential
 identities.
\end{lem}
 \begin{thm}\label{b3}
 \emph{(Kharchenko ~\cite{a3})}.
 Let $R$ be a prime ring, $d$
 a nonzero derivation of $R$ and $I$ a nonzero ideal of $R$.
 If $I$ satisfies the differential identity

 \noindent
 $$f(r_{1},r_{2},\ldots,r_{n},d(r_{1}),d(r_{2}),\ldots,d(r_{n})) = 0,$$
 for any $r_{1},r_{2},\ldots,r_{n}\in I$, then one of the following holds:

 \noindent
\begin{enumerate}[label=(\roman{*})]
   \item  $I$ satisfies the generalized polynomial identity
   $$f(r_{1},r_{2},\ldots,r_{n},x_{1},x_{2},\ldots,x_{n}) = 0.$$

   \item $d$ is $Q$-inner, that is, for some $q\in Q$,
   $d(x) = [q,x]$ and $I$ satisfies the generalized polynomial identity

  \noindent
  $$f(r_{1},r_{2},\ldots,r_{n},[q,r_{1}],[q,r_{2}],\ldots,[q,r_{n}]) = 0.$$
  \end{enumerate}
  \end{thm}

 \noindent
 We establish the following technical result required
 in the proof of Theorem \ref{b0}.
\begin{lem}\label{b4}
 Let $R$ be a prime ring with extended centroid $C$.
 Suppose $([a,x]y+x[a,y])^n-[a,x]^n[a,y]^n=0$, for all $x,y \in R$
 and some $a\in R$. Then $R$ is commutative or $a\in C$.
 \end{lem}
 \begin{proof}
 If $R$ is commutative there is nothing to prove.
 Suppose  $R$ is not commutative.
 Set
 $$f(x,y)=([a,x]y+x[a,y])^n-[a,x]^n[a,y]^n.$$
 Since $R$ is not commutative. Then by Lemma \ref{b1}, $f(x,y)$
 is a nontrivial generalized polynomial identity
 for $R$ and so for $Q$.

 \noindent
 In case $C$ is infinite, we have $f(x,y)=0$ for all
 $x, y\in Q\bigotimes_C \overline{C}$, where $\overline{C}$
 is the algebraic closure of $C$.
 Since both $Q$ and $Q\bigotimes_C\overline{C}$ are
 prime and centrally closed ~\cite{a5}, we may replace
 $R$ by $Q$ or $Q\bigotimes_C\overline{C}$ according to $C$ finite
 or infinite.
 Thus we may assume that $R$ is a centrally closed over $C$
 which is either finite or algebraically closed and $f(x,y)=0$
 for all $x, y\in R$.
 By Martindale's Theorem~\cite{a6}, $R$ is then a primitive
 ring having nonzero socle $H$ with $C$ as associated division ring.
 Hence by Jacobson's Theorem~\cite{a5} $R$ is isomorphic to
 a dense ring of linear transformations  of some vector space $V$
 over $C$, and $H$ consists of the linear transformations in $R$
 of finite rank.
 Let dim$_CV=k$. Then
 the density of $R$ on $V$  implies that $R\cong M_k(C)$.
 If  dim$_CV=1$, then $R$ is a commutative,
 which is a contradiction.

 \noindent
 Suppose that dim$_CV\geq 2$.
  We show that for any $v\in V$, $v$ and $av$ are linearly dependent
  over $C$. Suppose $v$ and $av$ are linearly independent for some
  $v\in V$.
  By density  of $R$, there exist $x, y \in R$ such that
  $$\begin{array}{cc}
    xv=0,  & xav=v, \\ \\
    yv=0,  & yav=v.
  \end{array}$$
  Since $[a,y]^nv=[a,x]^nv=(-1)^nv$.
  Hence we get following contradiction
  $$0=(([a,x]y+x[a,y])^n-[a,x]^n[a,y]^n)v=-v.$$
  So we conclude that
  $\{v, av\}$ are linearly
  $C$-dependent.
  Hence for each $v\in V$, $av=v\alpha_v$ for some
  $\alpha_v\in C$. Now we prove
  $\alpha_v$ is not depending on the choice of
  $v\in V$.

  \noindent
  Since dim$_CV\geq 2$ there exists
  $ w \in V$ such that $v$ and $w$ are  linearly
  independent over $C$. Now there exist
  $\alpha_v, \alpha_w, \alpha_{v+w}\in C$ such that
  $$av=v\alpha_v, aw=w\alpha_w, a(v+w)=(v+w)\alpha_{(v+w)}.$$
  Which implies
  $$v(\alpha_v-\alpha_{(v+w)})+w(\alpha_w-\alpha_{(v+w)})=0,$$
  and since $\{v, w \}$ are linearly $C$-independent,
  it follows $\alpha_v=\alpha_{(v+w)}=\alpha_w.$
  Therefore there exists $\alpha\in C$ such that
  $av=v\alpha$ for all $v\in V$.

 \noindent
 Now let $r\in R$, $v\in V$. Since $av=v\alpha$,
 $$[a,r]v=(ar)v-(ra)v=a(rv)-r(av)=(rv)\alpha-r(v\alpha)=0,$$
 that is $[a,r]V=0$.
 Hence $[a,r]=0$ for all $r\in R$, implying $a\in C$.
 \end{proof}

\smallskip
\noindent
 Now we can prove Theorem~\ref{b0}.

\vspace{3mm}
 \emph{Proof of Theorem~\ref{b0}.} Let $R$ be not commutative.
 By the given hypothesis, $R$ satisfies the generalized
 differential identity
 \begin{equation}\label{b4}
 (d(x)y+xd(y))^n=(d(x))^n(d(y))^n.
 \end{equation}
 By Lemma~\ref{b2}, $R$ and $Q$ satisfy the same differential
 identities, thus $Q$ satisfies (\ref{b4}).
 We divide the proof in two cases:\\

    \emph{Case 1}. $d$ is a $Q$-inner derivation.
  In the case, there exists an element $a\in Q$ such
  that $d(x)=[a,x]$ and $d(y)=[a,y]$ for all $x, y\in Q$.
  Notice that $Q$ satisfies the generalized polynomial identity
  $([a,x]y+x[a,y])^n=[a,x]^n[a,y]^n.$
  In this case the conclusion follows from Lemma~\ref{b4}.
  Thus we have $a\in C$ and so $d=0$.\\

 \emph{Case 2}. $d$ is not a $Q$-inner derivation. Applying
 Theorem~\ref{b2}, then (\ref{b4}) becomes
 $$(zy+xw)^n-(z)^n(w)^n,$$
 for all $x, y, z, w \in Q$. If $z=w$,
 then $Q$ satisfies
 $$(zy+xz)^n-z^{2n}=0.$$
 This is a polynomial identity. Hence there exists a field
 $F$ such that $Q\subseteq M_k(F)$, the ring of $k\times k$
 matrices over field $F$, where $k>1$.
 Moreover $Q$ and $M_k(F)$ satisfy the same polynomial identity
 ~\cite[Lemma 1]{a03}. Choose
 $$\begin{array}{cc}
   x=z=e_{ij}, & y=e_{ji},
 \end{array}$$
 for all $i\neq j$. This leads to the contradiction
 $$0=(zy+xz)^n-z^{2n}=e_{ii}.$$
 This complete the proof.\hfill $\Box$ \\

 \noindent
  The following example shows the hypothesis of primeness is essential in Theorem~\ref{b0}.
 \begin{example}
 Let $S$ be any ring, and
$R =\left\{{\small{\small{\small\left (
      \begin{array}{c c c}
        0 & a & b \\
        0 & 0 & c \\
        0 & 0 & 0
        \end{array} \right )}}}|a, b, c \in S \right\}.$
 Define $d:R\rightarrow R$ as follows:
$$d {\small{\small\left (
        \begin{array}{c c c}
        0 & a & b \\
        0 & 0 & c \\
        0 & 0 & 0
        \end{array} \right )}}=
 {\small{\small{\small\left (
        \begin{array}{c c c}
        0 & 0 & b \\
        0 & 0 & 0 \\
        0 & 0 & 0
        \end{array} \right )}}}.$$
 Then $0\neq d$ is a derivation of $R$ such that
 $(d(xy))^n=(d(x))^n(d(y))^n$ for all $x,y \in R$,
 where $n\geq 1$ is a fixed integer,
 however $R$ is  not commutative.
\end{example}

 Now let $R$ be a  semiprime ring.

 \noindent
 We establish the following technical result required
 in the proof of Theorem~\ref{b00}.

\begin{lem}\label{b5}
\emph{(see \cite[Lemma 1 and Theorem 1]{a01} or \cite[pages 31-32]{a4})}.
 Let $ R $ be a semiprime ring
 and $P$ a maximal ideal of $ C $.
 Then $PQ$ is a prime ideal of $Q$ invariant
 under all derivations of $Q$. Moreover
 $$\cap\{ P| PQ \text{ is maximal ideal of } C\}=0.$$

\end{lem}

\smallskip
\noindent
 Now we can prove Theorem~\ref{b00}.

\vspace{3mm}
\emph{Proof of Theorem~\ref{b00}.}  Since any derivation $d$
 can be uniquely extended to  a derivation in
 $Q$, and   $R$, $Q$ satisfy the same differential identities
 ~\cite[Theorem 3]{a4}, we have
 $$(d(xy))^n=(d(x))^n(d(y))^n,$$
 for all $x,y \in Q$.
 Let $ P $ be any maximal ideal of $C$ by Lemma~\ref{b5},
 $PQ$ is prime ideal of $Q$ invariant under $ d $.
 Set $\overline{Q}=Q/{PQ}$.
 Then derivation $ d $ canonically induces a derivation
 $\overline{d}$ on $\overline{Q}$ defined by
 $\bar{d}(\bar{x})=\overline{d(x)}$
 for all $ x \in Q$.
 Therefore,
 $$(\bar{d}(\overline{xy}))^n=(\bar {d}(\bar {x}))^n(\bar {d}(\bar {y}))^n,$$
 for all $\bar{x}, \bar{y} \in \overline{Q}$.
 By Theorem~\ref{b0} $d(Q)\subseteq PQ$ or $[Q,Q]\subseteq PQ$.
 Hence $d(Q)[Q,Q]\subseteq PQ$ for any maximal ideal $P$ of $C$.
 By Lemma~\ref{b5}, $d(Q)[Q,Q]=0$.
 Without loss of generality we have $d(R)[R,R]=0$.
 This implies that $$d(R^2)[R,R]=d(R)R[R,R].$$
 Therefore
 $$[R,d(R)]R[R,d(R)]=0.$$
 By semiprimeness of $ R $, we have  $[R,d(R)]=0 $.
 This complete the proof.\hfill $\Box$\\

 Now let $R$ be a semiprime orthogonally complete ring with
 extended centeroid $C$. The notations $B=B(C)$
 and spec$(B)$ denotes Boolian ring of $C$ and the
 set of all maximal ideal of $B$, respectively.
 It is well known that if $M\in$ spec$(B)$ then $R_{M}=R/RM$
 is prime ~\cite[Theorem 3.2.7]{a0}. We use the notations
 $\Omega$-$\Delta$-ring,
 Horn formulas and Hereditary formulas. We
 refer the reader to ~\cite[ pages 37, 38, 43, 120]{a0}
 for the definitions and the related properties
 of these objects.\\

\noindent
 We establish the following technical result required
 in the proof of Theorem \ref{b000}.

\begin{lem}\label{c}
 \cite[Theorem 3.2.18]{a0}. Let $R$ be an orthogonally complete
 $\Omega$-$\Delta$-ring with extended centroid $C$,
 $\Psi_{i} ( x_{1}, x_{2},\ldots, x_{n})$ Horn formulas of signature
 $\Omega$-$\Delta$, $i=1,2,\ldots$ and
 $\Phi(y_{1}, y_{2},\ldots,y_{m})$
 a Hereditary first order formula such that $\neg\Phi$ is
 a Horn formula.
 Further, let
 $\vec{a}= (a_{1}, a_{2},\ldots,a_{n})\in R^{(n)},$
 $\vec{c} = (c_{1}, c_{2},\ldots, c_{m})\in R^{(m)}.$
 Suppose  $R\models \Phi (\vec{c})$ and  for
 every $M\in$ spec $(B)$ there exists a natural number
 $i=i(M)>0$ such that
 $$R_{M} \models \Phi (\phi_{M} (\vec{c})) \Longrightarrow \Psi_{i}(\phi_{M}(\vec{a})),$$\\
  where $\phi_{M}: R\rightarrow R_{M} = R/RM$ is the
  canonical projection. Then there exists a natural number
  $k > 0$ and pairwise orthogonal idempotents
  $e_{1}, e_{2},\ldots,e_{k}\in B$ such that
  $e_{1} + e_{2} + \ldots + e_{k} = 1$ and $e_{i}R\models\Psi_{i}(e_{i}\vec{a})$
  for all $e_{i}\neq 0$.
 \end{lem}

 \noindent
 We denote $O(R)$ the orthogonal completion of $R$
 which is defined as the intersection of all
 orthogonally complete subset of $Q$ containing $R$.

\smallskip
\noindent
 Now we can prove Theorem~\ref{b000}.

\vspace{3mm}
\emph{Proof of Theorem~\ref{b000}.}
  By  assumption we have $R$  satisfies
 $$(d(xy))^n=(d(x))^n(d(y))^n.$$
 According to \cite[Theorem 3.1.16]{a0} $d(A)\subseteq A$
 and $d(e)=0$ for all $e\in B.$ Therefore,
 $A$ is an orthogonally complete $\Omega$-$\Delta$-ring,
 where $\Omega= \{o, +, -, \cdot, d \}$.
 Consider formulas

 $$\begin{array}{l}
 \Phi =  (\forall x )(\forall y )  \| (d(xy))^n=(d(x))^n(d(y))^n \|,\\ \\
 \Psi_{1} = (\forall x )  \| d(x)=0\|,\\ \\
 \Psi_{2} = (\forall x)(\forall y)\|xy=yx \|.
 \end{array}$$\\
 One can easily check that $\Phi $ is a hereditary first order
 formula and $\neg \Phi$, $\Psi_{1}$, $\Psi_{2}$ are Horn formulas.
 So using Theorem ~\ref{b0} shows that all conditions of lemma
 ~\ref{c} are fulfilled. Hence there exist two orthogonal
 idempotent $e_{1}$ and $e_{2}$ such that $e_{1} + e_{2} = 1$
 and if $e_{i} \neq 0$, then $e_{i}A \models \Psi_{i},$
 $i = 1, 2.$ The proof is complete.\hfill $\Box$

 \end{document}